\documentclass{amsart}

\newtheorem{thm}{Theorem}[section]

\newtheorem{cor}[thm]{Corollary}

\newtheorem{lemma}[thm]{Lemma}
\newtheorem{prop}[thm]{Proposition}

\theoremstyle{definition}
\newtheorem{defn}[thm]{Definition}
\newtheorem{remark}[thm]{Remark}

\newcommand{\bb}[1]{\mathbb{#1}}
\newcommand{\cal}[1]{\mathcal{#1}}

\newcommand{\Angle}[1]{\left\langle #1 \right\rangle}
\newcommand{\tr}{\mathrm{tr}}
\newcommand{\cl}[1]{\overline{#1}}
\newcommand{\Aut}{\mathrm{Aut}}

\newcommand\numberthis{\addtocounter{equation}{1}\tag{\theequation}}
\newcommand{\M}{\mathbb{M}}
\title{A synchronous game for binary constraint systems}

\author{Se-Jin Kim}
\address{Pure Mathematics Dept., University of Waterloo,
Waterloo, ON, N2L 3G1, Canada}
\email{s362kim@uwaterloo.ca}

\author{Vern Paulsen}
\address{Pure Mathematics Dept. and Institute for Quantum Computing, University of Waterloo,
Waterloo, ON, N2L 3G1, Canada}
\email{vpaulsen@uwaterloo.ca}
\author{Christopher Schafhauser}
\address{Pure Mathematics Dept., University of Waterloo,
Waterloo, ON, N2L 3G1, Canada}
\email{cschafhauser@uwaterloo.ca}
\thanks{The authors are partially supported by NSERC grants}

\begin{document}
\maketitle

\begin{abstract} Recently, W. Slofstra proved that the set of quantum correlations is not closed. We prove that the set of synchronous quantum correlations is not closed, which implies his result, by giving an example of a synchronous game that has a perfect quantum approximate strategy but no perfect quantum strategy. We also exhibit a graph for which the quantum independence number and the quantum approximate independence number are different. We prove new characterisations of synchronous quantum approximate correlations and synchronous quantum spatial correlations. We solve the synchronous approximation problem of Dykema and the second author, which yields a new equivalence of Connes' embedding problem in terms of synchronous correlations.

\end{abstract}
\section{Introduction}

W. Slofstra \cite{Slofstra17} solved one of the Tsirelson conjecture's by exhibiting a binary constraint system (BCS) game which has a perfect qa-strategy but no perfect q-strategy.  His result shows, in particular, that the set of probabilistic quantum correlations is not closed. In this paper, building on his work, we prove the somewhat stronger result that the set of synchronous quantum correlations is not closed.  We do this by constructing a synchronous game that has a perfect qa-strategy but no perfect q-strategy.  We also construct a graph $G$ with the property that $\alpha_q(G) < \alpha_{qa}(G)$, where these numbers are the quantum independence number and approximate quantum independence number of $G$, respectively.  These results rely on a major part of Slofstra's construction, but then use the ideas on quantum graph isomorphisms from \cite{graphisom}. To carry out this program, we need to extend the material in \cite{graphisom} on quantum graph isomorphisms to the cases of qa-strategies and qc-strategies.
The extension of the ideas of \cite{graphisom} to qc-strategies is minor, but the extension to qa-strategies builds on new characterisations of synchronous qa-strategies which are the main technical results of this paper. 

Section 2 contains many preliminary results and definitions, some background material on quantum correlations and synchronous games, and some remarks on the material in \cite{graphisom} to include other types of quantum isomorphisms.

Section 3 includes some new results on the connections between synchronous correlations and traces. In particular, we prove that the sets of synchronous q-correlations and synchronous qs-correlations are equal. We prove that in the correspondence between synchronous correlations and traces on C$^*$-algebras, established in \cite{estChromNo}, that synchronous qa-correlations arise from amenable traces. We also answer the {\it synchronous approximation problem} from \cite{DyPa}, which shows that Connes' embedding problem is equivalent to proving the equality of the set of synchronous qa-correlations with the set of synchronous qc-correlations. This in turn leads to another proof of Ozawa's \cite{Oz} equivalence of Connes' embedding problem with the equality of the set of all qa-correlations with the set of all qc-correlations.

In Section 4, we prove that the sets of synchronous q-correlations and synchronous qa-correlations are not equal. This implies Slofstra's \cite{Slofstra17} result that the set of quantum correlations is not closed, when the number of inputs and outputs is sufficiently large.

Finally, in Section 5, we return to the study of quantum versions of graph parameters and exhibit a graph for which $\alpha_q(G) < \alpha_{qa}(G)$.

\section{Preliminaries on Correlations and Quantum Graph Theory}

Suppose that Alice has $n_A$ quantum experiments each with $m_A$ outcomes and Bob has $n_B$ quantum experiments each with $m_B$ outcomes and that their combined labs are in some combined, possibly entangled, state.  We let $p(a,b|x,y)$ denote the conditional probability that if Alice conducts experiment $x$ and Bob conducts experiment $y$ then they get outcomes $a$ and $b$, respectively. The $n_An_Bm_Am_B$-tuple
\[\big( p(a,b|x,y) \big)_{1 \le x\le n_A, 1 \le y \le n_B, 1 \le a \le m_A, 1 \le b \le m_B}\]
of real numbers is, informally, called a {\it quantum correlation}.  There are several different mathematical models that can be used to describe these values, denoted by the subscripts, $q, qs, qa$ and $qc$ and the Tsirelson problems are concerned with whether or not these different mathematical models yield the same sets.  Due to the work of Slofstra \cite{Slofstra17}, we now know that in some cases these sets are different.

We now recall these sets formal definitions. For $t \in \{ q,qs, qa, qc \}$, we let $C_t(n_A, n_B ,m_A, m_B) \subseteq \bb R^{n_An_Bm_Am_B}$ denote the set of all possible tuples $\big( p(a,b|x,y) \big)$ that can be obtained using the model $t$.

In particular, $p(a,b|x,y) \in C_q(n_A, n_B,m_A, m_B)$ if and only if there exist finite dimensional Hilbert spaces $\cal H_A$ and $\cal H_B$, orthogonal projections $E_{x,a} \in B(\cal H_A), 1 \le x \le n_A, 1 \le a \le m_A$ satisfying $\sum_{a=1}^{m_A} E_{x,a} = I_{\cal H_A}, \, \forall x$,  orthogonal projections $F_{y,b} \in B(\cal H_B), 1 \le y \le n_B, 1 \le b \le m_B$ satisfying $\sum_{b=1}^{m_B} F_{y,b} = I_{\cal H_B}, \, \forall y$ and a unit vector $\psi \in \cal H_A \otimes \cal H_B$ such that
\[ p(a,b|x,y) = \langle E_{x,a} \otimes F_{y,b} \psi, \psi \rangle.\]

The set $C_{qs}(n_A,n_B,m_A,m_B)$ is defined similarly, except the condition that the Hilbert spaces $\cal H_A$ and $\cal H_B$ be finite dimensional is dropped.

It is known that the closures of these sets are the same and we denote the closure of these sets by $C_{qa}(n_A,n_B,m_A,m_B)$

The set $C_{qc}(n_A,n_B,m_A,m_B)$ is defined by eliminating the tensor product and instead having a single Hilbert space $\cal H$, a unit vector $\psi \in \cal H$, together with orthogonal projections $E_{x,a}, F_{y,b} \in B(\cal H)$ satisfying
\begin{enumerate}
  \item $E_{x,a}F_{y,b} = F_{y,b}E_{x,a}$ for all $a,b,x,y$,
  \item $\sum_{a=1}^{m_A} E_{x,a} = \sum_{b=1}^{m_B} F_{y,b} = I_{\cal H}$ for all $x, y$, and
  \item $p(a,b|x,y) = \langle E_{x,a}F_{y,b} \psi, \psi \rangle$ for all $a, b, x, y$.
\end{enumerate}
In each of the cases, i.e., for $t \in \{ q,qs,qa,qc \}$,
when $n_A=n_B =n$ and $m_A=m_B=m$, we set $C_t(n,m) = C_t(n,n,m,m).$

A correlation $\big( p(a,b|x,y) \big) \in C_t(n,m)$ is called {\it synchronous} provided that whenever $a \ne b, \,\, p(a,b|x,x) =0, \, \forall 1 \le x \le n$. We write $C_t^s(n,m)$ for the subset of synchronous correlations. Characterizations of synchronous correlations in terms of traces are known for the cases $t=q, qc$.  In Section 3, we give characterizations of synchronous correlations for the remaining cases, $t= qs, qa$.

By a {\it finite input-output game}, we mean a tuple $\cal G= (I_A, I_B, O_A, O_B, V)$ where $I_A, I_B, O_A, O_B$ are finite sets, representing the inputs that Alice and Bob can receive and the outputs that they can produce, respectively, and a function
\[ V: I_A \times I_B \times O_A \times O_B \to \{0,1 \} \]
called the {\it rule} or {\it predicate} function, where $V(x,y,a,b)=1$ means that if Alice and Bob receive $(x,y) \in I_A \times I_B$ and produce outputs $(a,b) \in O_A \times O_B$ then they win the game and if $V(x,y,a,b)=0$, then they lose the game.

A game is called {\it synchronous} provided that $I_A = I_B$, $O_A = O_B$ and the function $V$ satisfies $V(x,x,a,b) =0, \, \forall x, \forall a \ne b$.

Given a game, a correlation $\big( p(a,b|x,y) \big) \in C_t(|I_A|,|I_B|, |O_A|, |O_B|)$ is called a {\it perfect} or {\it winning} $t$-correlation, if the probability that it produces a losing output is 0, i.e., provided that
\[ V(x,y,a,b)=0 \implies p(a,b|x,y) =0.\]
When a game has a perfect $t$-correlation, then we say that the game possesses a perfect $t$-strategy.
Note that if a game is synchronous, then any perfect correlation must be synchronous.

From the definition of the set $C_{qa}(n_A, n_B, m_A, m_B)$ it readily follows that a game possesses a perfect qa-strategy if and only if for every $\epsilon >0$, there is a q-correlation $\big(p(a, b|x, y)\big)$ in $C_q(n_A, n_B, m_A, m_B)$ satisfying
\[ V(x,y,a,b)=0 \implies p(a,b|x,y) < \epsilon. \]

Every synchronous game $\cal G$ has a unital *-algebra $\cal A(\cal G)$ affiliated with it (possibly the zero algebra), defined by generators and relations. It has generators
\[ \{ E_{x,a}: 1 \le x \le n, 1 \le a \le m \} \]
satisfying the relations
\begin{enumerate}
  \item $E_{x,a}=E_{x,a}^*=E_{x,a}^2$ for all $a$ and $x$,
  \item $\sum_{a=1}^m E_{x,a} = I$ for all $x$, and
  \item for all $a$, $b$, $x$, and $y$, if $V(x,y,a,b)=0$, then $E_{x,a}E_{y,b} = 0$.
\end{enumerate}
One of the results of \cite{HMPS} is that a synchronous game $\cal G$ has a perfect q-strategy if and only if $\cal A(\cal G)$ has a unital *-representation as operators on a non-zero, finite dimensional Hilbert space.  Thus, a synchronous game $\cal G$ has a perfect q-strategy if and only if one can find projections $E_{x,a}$ on a finite dimensional Hilbert space satisfying the above relations for the given rule $V$.  Similarly, $\cal G$ has a perfect qc-strategy if and only if $\cal A(\cal G)$ has a unital *-representation into a C$^*$-algebra with a trace. The results of Section 3 will show that $\cal G$ has a perfect $qa$-strategy if and only if  $\cal A(\cal G)$ approximately has unital *-representations on non-zero, finite-dimensional Hilbert spaces; more precisely, $\cal A(\cal G)$ has a unital *-representation on $\cal R^\omega$, the tracial ultrapower of the hyperfinite $\mathrm{II}_1$-factor $\cal R$.
For readers not familiar with this ultrapower construction, more details can be found in \cite{BrownOzawa}.

There are two families of synchronous games, both involving graphs, that we wish to recall.

Let $G=(V,E)$ be a finite undirected graph without loops. That is,  $V$ is a finite set of {\it vertices} and $E \subseteq V \times V$ denotes the set of {\it edges}, $(v,v) \notin E, \, \forall v \in V$, since it contains no loops, and $(v,w) \in E \implies (w,v) \in E$, since it is undirected. We let $K_n$ denote the graph on $n$ vertices such that $(v,w) \in E, \, \forall v \ne w$. Given two graphs $G=(V(G), E(G))$ and $H=(V(H),E(H))$ by a {\it graph homomorphism from $G$ to $H$} we mean a function $f: V(G) \to V(H)$ satisfying
\[ (v,w) \in E(G) \implies (f(v), f(w)) \in E(H). \]
We write $\exists \, G \to H$ to indicate that there is a graph homomorphism from $G$ to $H$.

Many graph parameters can be defined in terms of graph homomorphisms. The {\it chromatic number} of $G$ is
\[\chi(G) = \min \{ c: \exists \, G \to K_c \}.\]
The {\it clique number} of $G$ is
\[ \omega(G) = \max \{ c: \exists \, K_c \to G \},\]
and the {\it independence number} of $G$ is
\[ \alpha(G)= \omega( \cl G),\]
where $\overline{G}=(V, \overline{E})$ denotes the {\it complement} of $G$; i.e., the graph with the same vertex set but for $v \ne w, \, (v,w) \in \overline{E} \iff (v,w) \notin E$.

Given graphs $G$ and $H$, the {\it graph homomorphism game from $G$ to $H$} is the synchronous game with inputs $V(G)$, outputs $V(H)$ and rule function
\[ V(v,w,x,y) = 0 \iff \big( (v,w) \in E(G) \,\, \text{and} \,\, (x,y) \notin E(H) \big) \,\, \text{or} \,\, \big( v=w \,\, \text{and} \,\, x \ne y \big).\]
For $t \in \{ q, qs, qa, qc \}$, we write $\exists \, G \stackrel{t}{\rightarrow} H$ to indicate that the graph homomorphism game from $G$ to $H$ has a perfect $t$-strategy.

In parallel with the above characterizations we set:
\[ \chi_t(G) = \min \{ c: \exists \, G \stackrel{t}{\to} K_c \}, \,\, \omega_t(G) = \max \{ c: \exists \, K_c \stackrel{t}{\to} G \}, \,\, \alpha_t(G) = \omega_t( \cl G).\]

It is not hard to verify that for complete graphs,
\[ \chi(K_n) = \chi_q(K_n) = \chi_{qs}(K_n)= \chi_{qa}(K_n) = \chi_{qc}(K_n) = n \]
and that
\[ \alpha(K_n)=\alpha_q(K_n) = \alpha_{qs}(K_n) = \alpha_{qa}(K_n) = \alpha_{qc}(K_n) = 1. \]
Indeed, by \cite{HMPS}, we have that
\[n = \chi(K_n) \ge \chi_q(K_n) \ge \chi_{qs}(K_n) \ge \chi_{qa}(K_n) \ge \chi_{qc}(K_n) \ge \chi_{hered}(K_n) = n.\]
The fact that for $t \in \{ q, qs, qa, qc \}$,  $\chi_t(K_n)=n$ implies that if there exists $K_n \stackrel{t}{\to} K_c$ then $n \le c$.
This in turn implies that
\[ c= \alpha(K_c) \le \alpha_q(K_c) \le \alpha_{qs}(K_c) \le \alpha_{qa}(K_c) \le \alpha_{qc}(K_c) \le c,\]
where the last inequality follows since $\alpha_{qc}(K_c)$ is the largest $n$ for which $K_n \stackrel{qc}{\to} K_c$.

The second game that we shall need is the {\it $(G, H)$-isomorphism game} defined in \cite{graphisom}.  This game is intended to capture the concept of two graphs being isomorphic. It is a synchronous game with input set and output set both equal to $V(G) \cup V(H)$ where we view the vertex sets as disjoint. We refer the reader to \cite{graphisom} for the rules of this game. For $t \in \{ q, qs, qa, qc \}$ we write $G \cong_t H$ to indicate that there is a perfect $t$-strategy for the $(G,H)$-isomorphism game. In \cite{graphisom}, they only introduced and studied the cases $t=q$ and $t=ns$ (which we have not introduced here).

However, we shall use the fact that since this is a synchronous game, it will have an affiliated *-algebra with generators and relations that can be used to characterize when perfect $t$-strategies exist.  In fact, the generators and relations for the *-algebra of the game are precisely the relations $(IQP_d)$ in \cite{graphisom}.  We now recall the *-algebra $\cal A(\cal G)$ corresponding to the $(G, H)$-isomorphism game $\cal G$.  First we need some notation.  Given vertices $g, g' \in V(G)$ and $h, h' \in V(H)$, write $\operatorname{rel}(g, g') = \operatorname{rel}(h, h')$ if any of the following hold:
\begin{enumerate}
  \item $g = g'$ and $h = h'$;
  \item $(g, g') \in E(G)$ and $(h, h') \in E(H)$;
  \item $g \neq g'$, $(g, g') \notin E(G)$, $h \neq h'$, and $(h, h') \notin E(H)$.
\end{enumerate}
The *-algebra $\cal A(\cal G)$ is generated by elements
\[ \{ X_{g, h} : g \in V(G), h \in V(H) \} \]
subject to the relations
\[ X_{g, h} = X_{g, h}^* = X_{g, h}^2, \quad \text{and} \quad \sum_{h' \in V(H)} X_{g, h'} = \sum_{g' \in V(G)} X_{g', h} = 1 \]
for all $g \in V(G)$ and $h \in V(H)$ and
\[ \operatorname{rel}(g, g') \neq \operatorname{rel}(h, h') \quad \Rightarrow \quad X_{g, h} X_{g', h'} = 0 \]
for all $g, g' \in V(G)$ and $h, h' \in V(H)$.


We end this preliminary section with some notation that will be used throughout.  Let $\bb F(n, m)$ denote the group freely generated by $n$ elements of order $m$ and let $\mathrm{C}^*(\bb F(n, m))$ denote the universal group C$^*$-algebra of $\bb F(n, m)$.  For $x = 1, \ldots, n$, let $u_x$ be the unitary in $\mathrm{C}^*(\bb F(n, m))$ corresponding to the $x$th generator of $\bb F(n, m)$.  If $\omega_m$ denotes a primitive $m$th roots of unity, the spectral values of $u_x$ are $\omega_m^i$ for $i = 1, \ldots, m$.  Let $e_{x, i}$ denote the spectral projection of $u_x$ at the spectral value $\omega_m^i$.  Then we have $e_{x, i}$ is a projection for all $x$ and $i$ and $\sum_i e_{x, i} = 1$ for all $x$.

Conversely, a C$^*$-algebra $\cal A$ and projections $e_{x, i} \in \cal A$ for $1 \leq x \leq n$ and $1 \leq i \leq n$ such that $\sum_i e_{x, i} = 1$ for all $x$, the element $v_x = \sum_i \omega_m^i e_{x, i}$ is a unitary in $\cal A$ with order $m$.  Hence there is a unique *-homomorphism $\mathrm{C}^*(\bb F(n, m)) \rightarrow \cal A$ determined by $u_i \mapsto v_i$.  A straight forward calculation shows these constructions are inverses of each other and hence $\mathrm{C}^*(\bb F(n, m))$ is the universal C$^*$-algebra generated by projections $e_{x, i}$ for $1 \leq x \leq n$ and $1 \leq j \leq m$ such that $\sum_i e_{x,i} = 1$ for all $x$.

\section{Characterizations of Synchronous strategies}

In \cite{estChromNo} it was shown that synchronous quantum strategies arise from various families of traces. In particular, it was shown that $p(i,j|v,w) \in C^s_{qc}(n,m)$ if and only if there is a tracial state $\tau: C^*(\bb F(n,m)) \rightarrow \bb C$ such that $p(i,j|v,w) = \tau(e_{v,i}e_{w,j})$ and $p(i,j|v,w) \in C^s_q(n,m)$ if and only if there was a tracial state as before such that in addition the GNS representation of $\big( C^*(\bb F(n,m)), \tau \big)$ is finite dimensional. But at the time no characterization were given of the traces that arise from synchronous quantum spatial correlations or synchronous quantum approximate correlations.  In this section we provide characterizations of those two types of traces.

\begin{defn}\label{defn:amenableTrace}
Let $\cal A \subseteq B(\cal H)$ be a C$^*$-algebra.  A tracial state $\tau$ on $\cal A$ is called {\it amenable} provided there is a state $\rho$ on $B(\cal H)$ such that $\rho|_{\cal A} = \tau$ and $\rho(u T u^*) = \rho(T)$ for all $T \in B(\cal H)$ and all unitaries $u \in \cal A$.
\end{defn}

By an application of Arveson's Extension Theorem, the amenability of $\tau$ is independent of the choice of faithful representation of $\cal A$.  The following is due to Kirchberg in \cite[Proposition 3.2]{Kirchberg:AmenableTraces} (see also Theorem 6.2.7 in \cite{BrownOzawa}).  Here $\cal R$ denotes the hyperfinite $\text{II}_1$-factor, $\omega$ is a free ultrafilter over the positive integers, $\cal R^{\omega}$ is the corresponding tracial ultrapower.  See Appendix A in \cite{BrownOzawa} for the relevant definitions.

\begin{thm}\label{amenableTrace}
Suppose $\cal A$ is a separable C*-algebra and $\tau$ is a tracial state on $\cal A$.  The following are equivalent:
\begin{enumerate}
  \item the tracial state $\tau$ is amenable;
  \item there is a *-homomorphism $\varphi : \cal A \rightarrow \cal R^\omega$ with a completely positive, contractive lift $\cal A \rightarrow \ell^\infty(\cal R)$ such that $\tr \circ \varphi = \tau$;
  \item there is a sequence of completely positive, contractive maps $\varphi_k : A \rightarrow \bb M_{d(k)}$ such that
  \[ \| \varphi_k(ab) - \varphi_k(a)\varphi_k(b) \|_2 \rightarrow 0 \quad \text{and} \quad \operatorname{tr}_{d(k)}(\varphi_k(a)) \rightarrow \tau(a) \]
  for all $a, b \in \cal A$;
  \item the linear functional $\phi: \cal A \otimes \cal A^{op} \rightarrow \bb C$ defined by $\phi(a \otimes b^{op})= \tau(ab)$ is bounded with respect to the minimal tensor product;
\end{enumerate}
\end{thm}

In condition (4), note that if $\phi$ is bounded then for any $x = \sum_i a_i \otimes b_i^{op}$ we have that
\[ \phi(x^*x) = \sum_{i,j} \tau(a_j^*a_ib_ib_j^*) = \sum_{i,j} \tau((a_ib_i)(a_jb_j)^*) \ge 0.\]
Since $\phi(1 \otimes 1) =1$, we see that if $\phi$ is bounded, then $\phi$ is a state.

Recall that $\mathrm{C}^*(\bb F(n, m))$ is generated by a set of $n$ unitaries, $u_v, 1 \le v \le n$, of order $m$ and $e_{v, i}$ denotes the spectral projection of $u_v$ corresponding to the spectral value $\omega_m^i$ where $\omega_m$ is a primitive $m$th root of unity.

\begin{lemma} There is a *-isomorphism $\gamma: \mathrm{C}^*(\bb F(n,m)) \rightarrow \mathrm{C}^*(\bb F(n,m))^{op}$ with $\gamma(u_v^j) = u_v^j, 1 \le v \le n, \, 1 \le j \le m-1$.  Moreover, $\gamma(e_{v,i}) = e_{v,i}$ for $1 \leq v \leq n$ and $1 \leq i \leq m$.
\end{lemma}

\begin{proof} The words of the form
\[ u_{v_1}^{n_1} \cdots u_{v_K}^{n_K}\]
span a dense *-subalgebra of $\mathrm{C}^*(\bb F(n,m))$.  If we set
\[ \gamma( u_{v_1}^{n_1} \cdots u_{v_K}^{n_K}) = u_{v_K}^{n_K} \cdots u_{n_1}^{n_1},\]
and extend linearly, then it is easily checked that $\gamma$ extends to the desired *-isomorphism.  The second claim is a simple computation.
\end{proof}

\begin{lemma}\label{distToSpectralProj}
Suppose $n \geq 1$ and $p \in \mathbb{M}_n$ is a positive contraction.  If $q$ denotes the spectral projection of $p$ for the interval $[1/2, 1]$, then
\[ \| p - q \|_2 \leq 2 \sqrt{2} \| p - p^2\|_2. \]
\end{lemma}

\begin{proof}
Define $p_0 = (1-q)p$ and $p_1 = qp$.  Note that $\|p_i - p_i^2\|_2 \leq \|p - p^2\|_2$ for $i = 0, 1$.  Since $0 \leq p_0 \leq \frac12$, we have
\[ p_0 - p_0^2 = p_0(1 - p_0) \geq \frac12 p_0 \]
and hence $\|p_0\|_2 \leq 2 \|p_0 - p_0^2\|_2$.  Similarly, since $\frac12 q \leq p_1 \leq 1$, we have
\[ p_1 - p_1^2 = p_1(1 - p_1) \geq \frac12 q (1 - p_1) = \frac12 (q - p_1) \]
and hence $\|p_1 - q\|_2 \leq 2 \|p_0 - p_0^2\|_2$.  Since $p_0$ and $p_1 - q$ are orthogonal, the result follows from the Pythagorean identity.
\end{proof}

\begin{lemma}\label{stabilityOfCorrelations}
Given $\varepsilon > 0$ and an integer $m \geq 1$, there is a $\delta > 0$ such that for any integer $d \geq 1$, if $p_1, \ldots, p_m \in \mathbb{M}_d$ are positive contractions with $\|p_i^2 - p_i\|_2 < \delta$ and $\|p_i p_j\|_2 < \delta$ for all $i, j = 1, \ldots, m$ with $i \neq j$, then there are mutually orthogonal projections $q_1, \ldots, q_m \in \mathbb{M}_d$ such that $\|p_i - q_i\|_2 < \varepsilon$ for all $i = 1, \ldots, m$.

If in the statement above, we further require $\| \sum_i p_i - 1 \|_2 < \delta$, then we may arrange for $\sum_i q_i = 1$.
\end{lemma}

\begin{proof}
We prove the result by induction on $m$.  When $m = 1$, this is immediate Lemma \ref{distToSpectralProj}.  Assume the result holds for an integer $m \geq 1$.  Fix $\varepsilon > 0$ and define $\varepsilon_0 = \varepsilon/(40 m + 3)$.   Let $\delta_0 > 0$ be the constant obtained by applying the current lemma to $m$ and $\varepsilon_0$ and define $\delta := \min\{\delta_0, \varepsilon_0 \}$.  Suppose $d \geq 1$ and $p_1, \ldots, p_{m+1} \in \mathbb{M}_d$ are positive contractions as above.  By the choice of $\delta$, there are mutually orthogonal projections $q_1, \ldots, q_m \in \mathbb{M}_d$ such that
\[ \|p_i - q_i\|_2 < \varepsilon_0 < \varepsilon. \]
Since $\|p_i p_{m+1}\|_2 < \delta$ for all $i = 1, \ldots, m$, we have
\[ \|q_i p_{m+1}\|_2 < \varepsilon_0 + \delta < 2 \varepsilon_0. \]
Define $r = (1 - q_1 - \cdots q_m)$ and define $p = p_{m+1}$.  Then
\[ \| r p r - p \|_2 \leq  2 \| r p - p \|_2 \leq 2 \sum_{i=1}^m \| q_i p \|_2 < 4 m \varepsilon_0. \]
Now, note that
\begin{align*}
  \| (r p r)^2 - r p r \|_2 &= \| (r p r)^2 - p^2 \|_2 + \| p^2 - p \|_2 + \|p - r p r\|_2 \\ &\leq 3 \| r p r - p \|_2 + \|p^2 - p\|_2 < 12 m \varepsilon_0 + \delta < (12 m + 1) \varepsilon_0.
\end{align*}
By the previous lemma, if $q_{m+1}$ denotes the spectral projection of $r p r$ corresponding to the interval $[1/2, 1]$, then
\[ \| q_{m+1} - r p r \|_2 < 2\sqrt{2} (12 m + 1) \varepsilon_0 . \]
Therefore,
\[ \| q_{m + 1} - p \|_2 < 2 \sqrt{2} (12 m + 1)\varepsilon_0 + 4 m \varepsilon_0  < (40 m + 3)\varepsilon_0 = \varepsilon. \]
Note that each of the projections $q_1, \ldots q_m$ are orthogonal to $r$ by construction.  As $q_{m+1}$ is a spectral projections of $r p r$, we also have that each of the projections $q_1, \ldots, q_m$ is orthogonal to $q_{m+1}$.  This completes the proof of the first part of the lemma.

To see the final sentence holds, fix $m \geq 1$ and $\varepsilon > 0$.  Let $\varepsilon_0 = \varepsilon / (m + 2)$ and let $\delta_0$ be the constant given by applying the first part of the lemma to $m$ and $\varepsilon_0$.  Define $\delta = \min \{ \varepsilon_0, \delta_0 \}$.  Suppose $d \geq 1$ and $p_1, \ldots, p_m \in \mathbb{M}_d$ are projections such that
\[ \| p_i - p_i^2 \|_2 < \delta, \quad \| p_i p_j \|_2 < \delta, \quad \text{and} \quad \big\|\sum_k p_k - 1 \big\|_2 < \delta \]
for all $i, j = 1, \ldots, m$ with $i \neq j$.  By the choice of $\delta$, there are mutually orthogonal projections $q_1', q_2, q_3, \ldots q_m \in \mathbb{M}_d$ such that $\|p_1 - q_1'\|_2 < \varepsilon_0$ and $\|p_i - q_i\|_2 < \varepsilon_0$ for $i = 2, \ldots, m$.  Now, define $q_1'' = 1 - q_1' - \sum_{i=2}^m q_i$ and note that $\|q_1''\|_2 < (m+1)\varepsilon_0$.  To complete the proof, define $q_1 = q_1' + q_1''$.
\end{proof}

\begin{thm}\label{syncapproxthm}
Fix integers $n, m \geq 1$.  For $\big(p(i, j | v, w)\big) \in \bb R^{n^2m^2}$, the following are equivalent:
\begin{enumerate}
  \item $\big(p(i, j |v, w)\big) \in C_{qa}^s(n, m)$;
  \item there are synchronous correlations $\big(p_k(i, j | v, w)\big) \in C_q^s(n, m)$ with
  \[ p_k(i, j | v, w) \rightarrow p(i, j|v, w) \quad \forall \,\, i, j, v, w; \]
  \item there is an amenable trace $\tau$ on $\mathrm{C}^*(\mathbb{F}(n, m))$ such that
  \[ \tau(e_{v, i} e_{w, j}) = p(i, j | v, w) \quad \forall \, \, i, j, v, w; \]
  \item there are projections $f_{v, i} \in \cal R^\omega$ such that $\sum_i f_{v, i} = 1$ for all $v$ and
  \[ \tr(f_{v, i} f_{w, j}) = p(i, j | v, w) \quad \forall \, \, i, j, v, w. \]
\end{enumerate}
\end{thm}

\begin{proof}
It is clear that (2) implies (1).  To see (1) implies (3), assume that $\big(p(i,j|v,w)\big)$ is a correlation in $C_{qa}^s(n,m)$.  There exist correlations $\big(p_k(i,j|v,w)\big)$ in $C_q(n,m)$ for $k \geq 1$ such that
\[ \lim_k p_k(i,j|v,w) = p(i,j|v,w) \quad \forall \,\, i, j, v, w. \]
Each $\big(p_k(i, j, v, w) \big)$ has a representation on a tensor product of finite dimensional vector spaces  $\bb C^{d_k} \otimes \bb C^{r_k}$ as
\[ p_k(i,j|v,w) = \langle E^k_{v,i} \otimes F^k_{w,j}  \psi_k , \psi_k \rangle,\]
where the matrices $E^k_{v,i}, F^k_{w,j}$ are all orthogonal projections satisfying $\sum_i E^k_{v,i} = I_{d_k}$ and $\sum_j F^k_{w,j} = I_{r_k}$ and each $\psi_k$ is a unit vector.

Thus there is a representation $\pi_k: C^*(\bb F(n,m)) \otimes C^*(\bb F(n,m))^{op} \rightarrow \bb M_{d_k} \otimes \bb M_{r_k}$ with
$\pi_k(e_{v,i} \otimes e_{w,j}^{op}) = E^k_{v,i} \otimes F^k_{w,j}$.  Setting $\phi_k(a \otimes b^{op}) = \langle \pi_k(a \otimes b^{op}) \psi_k, \psi_k \rangle$ defines a sequence of states $\phi_k$ on $C^*(\bb F(n,m) \otimes C^*(\bb F(n,m))^{op}$.  Let $\phi$ be any weak*-limit point of $(\phi_k)_k$ and note that $\phi(e_{v,i} \otimes e_{w,j}^{op})= p(i,j|v,w)$.

If we let  $\pi: C^*(\bb F(n,m)) \otimes_{min} C^*(\bb F(n,m))^{op} \rightarrow B(\cal H)$ and $\psi \in H$ be a GNS representation of this state, then it follows by \cite[Theorem~5.5]{estChromNo}, that $\tau(a) = \langle \pi(a \otimes 1)\psi, \psi \rangle$ is a trace and that
\[\pi(a \otimes e_{w,j})\psi= \pi(ae_{w,j} \otimes 1)\psi.\]
Hence,  $\pi(a \otimes b e_{w,j})\psi = \pi(1 \otimes b)\pi(a \otimes e_{w,j})\psi = \pi(ae_{w,j} \otimes b) \psi$ and it follows that
\[ \phi(a \otimes b^{op}) = \langle \pi(a \otimes b^{op})\psi,\psi \rangle = \langle \pi(ab \otimes 1) \psi, \psi \rangle = \tau(ab).\]
Thus, $\tau$ is an amenable trace by Theorem \ref{amenableTrace}.

To see (3) implies (2), it suffices to show that if $\tau$ is an amenable trace on $\mathrm{C}^*(\mathbb{F}(n, m))$, then there is a sequence of traces $\tau_k$ on $\mathrm{C}^*(\mathbb{F}(n, m))$ which factor through a finite dimensional matrix algebra such that $\tau_k(a) \rightarrow \tau(a)$ for all $a \in \mathrm{C}^*(\mathbb{F}(n, m))$.  Since $\tau$ is amenable, Theorem \ref{amenableTrace} yields a sequence of completely positive, unital maps $\varphi_k : \mathrm{C}^*(\mathbb{F}(n, m)) \rightarrow \mathbb{M}_{d(k)}$ such that
\[ \| \varphi_k(ab) - \varphi_k(a)\varphi_k(b) \|_2 \rightarrow 0 \quad \text{and} \quad \operatorname{tr}(\varphi_k(a)) \rightarrow \tau(a) \]
for all $a, b \in \mathrm{C}^*(\mathbb{F}(n, m))$.  By passing to a subsequence and applying Lemma \ref{stabilityOfCorrelations}, we may find projections $p_{v, i}^k \in \mathbb{M}_{d(k)}$ such that $\sum_i p_{v, i}^k = 1$ for all $v$ and $k$ and such that $\| \varphi_k(e_{v, i}) - p_{v, i}^k \|_2 \rightarrow 0$ for all $v$ and $i$.  There is a *-homomorphism $\varphi'_k : \mathrm{C}^*(\mathbb{F}(n, m)) \rightarrow \mathbb{M}_{d(k)}$ such that $\varphi'_k(e_{v, i}) = p_{v, i}^k$ for all $v$, $i$, and $k$.  Using that $\mathrm{C}^*(\mathbb{F}(n, m))$ is generated as a C$^*$-algebra by the projections $e_{v, i}$, one can show
\[ \| \varphi_k(a) - \varphi'_k(a) \|_2 \rightarrow 0 \]
for all $a \in \mathrm{C}^*(\mathbb{F}(n, m))$.  In particular,
\[ \lim \operatorname{tr}(\varphi'_k(a)) = \lim \operatorname{tr}(\varphi_k(a)) = \tau(a) \]
for all $a \in \mathrm{C}^*(\mathbb{F}(n, m))$.

To see (3) implies (4), note that if $\tau$ is an amenable trace on $\mathrm{C}^*(\mathbb{F}(n, m))$, then there is a trace preserving *-homomorphism $\varphi : \mathrm{C}^*(\mathbb{F}(n, m)) \rightarrow \mathcal{R}^\omega$ by Theorem \ref{amenableTrace}.  Define $f_{v, i} = \varphi(e_{v, i}) \in \cal R^\omega$ for all $v$ and $i$.  Conversely, given $f_{v, i}$ as in (4), there is a *-homomorphism $\varphi : \mathrm{C}^*(\mathbb{F}(n, m)) \rightarrow \cal R^\omega$ such that $\varphi(e_{v, i}) = f_{v, i}$ for all $v$ and $i$.  As $\mathrm{C}^*(\mathbb{F}(n, m))$ has the local lifting property, $\varphi$ has a completely positive, unital lift $\mathrm{C}^*(\mathbb{F}(n, m)) \rightarrow \ell^\infty(\cal R)$.  By Theorem \ref{amenableTrace}, the trace $\tau := \tr \circ \varphi$ on $\mathrm{C}^*(\mathbb{F}(n, m))$ is amenable.
\end{proof}

\begin{cor}
Let $\cal G=(I,O,V)$ be a synchronous game.  Then the following are equivalent:
\begin{itemize}
\item[(i)] $\cal G$ has a perfect qa-strategy,
\item[(ii)] there is a unital *-representation of $\cal A(\cal G)$ into $\cal R^\omega$,
\item[(iii)] there is an amenable trace $\tau$ on $\mathrm{C}^*(\mathbb{F}(n, m))$ such that
  \[ V(v,w,i,j)=0 \implies \tau(e_{v, i} e_{w, j}) = 0 \quad \forall \, \, i, j, v, w. \]
  \end{itemize}
\end{cor}

\begin{cor} The following are equivalent:
\begin{itemize}
  \item[(i)] Connes' embedding conjecture has an affirmative answer,
  \item[(ii)] $C^s_{qa}(n,m) = C^s_{qc}(n,m), \, \forall n,m$,
  \item[(iii)] $C_{qa}(n,m) = C_{qc}(n,m), \, \forall n,m$.
\end{itemize}
\end{cor}

\begin{proof}  The equivalence of (i) and (ii) in Theorem~\ref{syncapproxthm} answers \cite[Problem ~3.8]{DyPa}.  In the remarks following Problem~3.8, \cite{DyPa} shows how a positive solution of the problem leads to the above result.
\end{proof}

\begin{remark} The implication $(iii)\implies (i)$ in the above corollary is due to Ozawa~\cite{Oz}. The equivalence of (i) and (ii) follows from \cite[Theorem~3.7]{DyPa} and our solution of their synchronous approximation problem. Note that the implication (iii) implies (ii) is trivial, so we have a different proof of Ozawa's implication. Ozawa's proof uses Kirchberg's results showing the equivalence of Connes' embedding conjecture to the equality of the minimal and maximal tensor products of certain C$^*$-algebras of free groups. The above proof uses the results of \cite{DyPa} which in turn used Kirchberg's results about the equivalence of Connes' embedding conjecture to finite approximability of traces, often referred as the matricial microstates conjecture.
\end{remark}

We next turn our attention to the set of synchronous quantum spatial correlations.  We prove the somewhat surprising result that any synchronous correlation that that can be obtained using a tensor product of possible infinite dimensional Hilbert spaces has a representation using only finite dimensional spaces.

\begin{thm} Let $n,m \in \bb N$.  Then $C^s_q(n,m) = C^s_{qs}(n,m)$.
\end{thm}
\begin{proof}  Clearly, $C_q^s(n,m) \subseteq C^s_{qs}(n, m)$, so we must prove that $C^s_{qs}(n,m) \subseteq C^s_q(n,m)$.  Let $\big( p(i,j|v,w) \big) \in C^s_{qs}(n,m)$ be represented as
\[ p(i,j|v,w) = \langle E_{v.i} \otimes F_{w,j} \psi, \psi \rangle \]
where $\{ E_{v,i}, 1 \le v \le n, 1 \le i \le m \}$ are orthogonal projections on some Hilbert space $\cal H$ satisfying $\sum_i E_{v,i} = I_{\cal H}, \forall v$, $\{ F_{w,j} : 1 \le w \le n, 1 \le j \le m \}$ are orthogonal projections on some Hilbert space $\cal K$ satisfying $\sum_j F_{w,j} = I_{\cal K}, \forall w$ and $\psi \in \cal H \otimes \cal K$ is a unit vector.

Note that if we are given any other Hilbert space $\cal G$ and we set $F_{w,1}^{\prime} = F_{w,1} \oplus I_{\cal G}$ and $F_{w,j}^{\prime} = F_{w,j} \oplus 0$, then
$p(i,j|v,w) = \langle (E_{v,i} \otimes F_{w,j}^{\prime}) \psi, \psi \rangle$.  In this manner we see that there is no loss of generality in assuming that $dim(\cal H) = dim(\cal K)$, so we assume that these two Hilbert spaces have the same dimension.

Let $\sum_{k \in K} \alpha_k e_k \otimes f_k$ be the Schmidt decomposition of $\psi$ so that $K$ is a countable set and $\{e_k: k \in K \}$ and $\{ f_k: k \in K \}$ are orthonormal sets in their respective Hilbert spaces.  By setting sufficiently many $\alpha$'s equal to 0, and direct summing with additional Hilbert spaces as needed, we may assume that these sets are orthonormal bases for their respective spaces.

Let $\{ r_l: l \in L \} = \{ \alpha_k: k \in K \}$ be an enumeration of the set of distinct non-zero $\alpha_k$'s (which is at most countable) with $r_1 \ge r_2 \ge \ldots$ and let $S_l = \{ k: \alpha_k = r_l \}.$  Let
$\cal E_l = span \{ e_k: k \in S_l \}$ and $\cal F_l = span \{ f_k: k \in S_l \}.$ Since the $\alpha_k$'s are square summable, each set $S_l$ is finite and so each of these spaces is finite dimensional.

We claim that the spaces $\cal E_l$ are reducing subspaces for $\{ E_{v,i} \}$ and that the spaces $\cal F_l$ are reducing for  the set  $\{ F_{w,j} \}$

First, we complete the proof assuming the claim. Let $E^l_{v,i}$ denote the compression of $E_{v,i}$ to the space $\cal E_l$ and let $F^l_{w,j}$ denote the compression of $F_{w,j}$ to the space $\cal F_l$ so that these are orthogonal projections and $\sum_i E^l_{v,i} = I_{\cal E_l}, \forall v$ and $\sum_j F^l_{w,j} = I_{\cal F_l}, \forall e$.  Set $d_l = dim(\cal E_l) = dim(\cal F_l) = card( S_l )$ and let $\psi_l = \frac{1}{\sqrt{d_l}} \sum_{k \in S_l} e_k \otimes f_k \in \cal E_l \otimes \cal F_l$, which is a unit vector.
Let $t_l = \frac{r_l^2}{d_l}$ so that $\sum_l t_l =1$ and set
\[ p_l(i,j|v,w) = \langle E^l_{v,i} \otimes F^l_{w,j} \psi_l, \psi_l \rangle \in C_q(n,m).\]
Note that
\[ \sum_l t_l p_l(i,j|v,w) = p(i,j|v,w),\]
so that for $i \ne j$,  $\sum_l t_l p_l(i,j|v,v) = 0$ from which it follows that $p(i,j|v,v)=0, \forall l$.  Thus, each $p_l(i,j|v,w) \in C^s_q(n,m)$.

Since $C^s_q(n,m)$ is convex, by \cite{Cook72}, $p(i,j|v,w) \in C^s_q(n,m)$. The key point here is that by \cite{Cook72} a convex set need not be closed to ensure that such a series remains in the set.

Thus, we need only establish that these spaces reduce the operators.  Let $\omega= e^{2 \pi i/m}$ be a primitive $m$-th root of unity and let $A_v= \sum_{i=1}^m \omega^i E_{v,i}$ and let $B_w= \sum_{j=1}^m \omega^j F_{w,j}$ so that these are unitaries of order $m$ and the original projections are the spectral projections of these unitaries.  Note that these unitaries generate the same C$^*$-algebras as the projections so that the projections are reduced by these subspaces if and only if these unitaries are reduced by these subspaces.

First recall that the synchronous condition guarantees that $(E_{v,i} \otimes I) \psi = (I \otimes F_{v,i}) \psi$ by \cite[Theorem~5.5i]{estChromNo}  and hence,  $(A_v \otimes I) \psi = (I \otimes B_v) \psi$.

Now compute that $(A \otimes I) \psi = (I \otimes B) \psi$ implies
\[ \alpha_j \langle Ae_j, e_i \rangle = \langle (A \otimes I) \psi, e_i \otimes f_j \rangle = \langle (I \otimes B) \psi, e_i \otimes f_j \rangle = \alpha_i \langle Bf_i, f_j \rangle .\]
Thus for $i \in S_1$, using that $\alpha_1 \ge \alpha_j$, we have
\begin{multline*}
|\alpha_1|^2 \ge  \sum_j |\alpha_j|^2 |\langle A_v e_j, e_i \rangle|^2 = \sum_j |\alpha_i|^2 |\langle B_vf_i, f_j \rangle|^2 = |\alpha_1|^2 \|B_vf_i\|^2 = |\alpha_1|^2,
\end{multline*}
and so we must have equality throughout. But equality implies that $\langle A_ve_j, e_i \rangle =0,  \forall j \notin S_1$. Hence, $A_v^*e_i \in \cal E_1, \forall i \in S_1$.  This shows that $A_v^*$ leaves $\cal E_1$ invariant. Hence,  $A_v= \big(A_v^*\big)^{m-1}$ also leaves this space invariant and so $\cal E_1$ is a reducing subspace for every $A_v$ and hence for the entire C$^*$-algebra that they generate.  A similar proof shows that $\cal F_1$ is reducing for every $B_v$.

Now it follows that for $i \in S_2$, we have that for $j \in S_1,  \langle A_ve_j, e_i \rangle =0$ and so,
\[ |r_2|^2 \ge \sum_j |\alpha_j|^2 |\langle A_v e_j, e_i \rangle|^2 = \sum_j |r_2|^2 |\langle B_vf_i, f_j \rangle|^2 = |r_2|^2, \]
and similar reasoning shows that $A_v^*e_i \in \cal E_2$ and consequently, that $\cal E_2$ reduces these unitaries.

The rest of the proof now follows by induction.
\end{proof}

\begin{cor}
A synchronous game has a perfect qs-strategy if and only if it has a perfect q-strategy.
\end{cor}

\section{Separating $C^s_{qs}$ and $C^s_{qa}$}

Suppose $Ax = b$ is an $m \times n$ linear system over $\mathbb{Z}/2$; that is, $A = (a_{i, j}) \in \mathbb{M}_{m, n}(\mathbb{Z}/2)$ and $b \in (\mathbb{Z}/2)^n$.  Let $V_i = \{ j \in \{1, \ldots, n \} : a_{i, j} \neq 0 \}$ denote the variables which occur in the $i$th equation for $i = 1, \ldots, m$.  It will be convenient to write the system multiplicative notation where we identify $\mathbb{Z}/2$ with $\{\pm 1\}$ and write the $i$th equation of the linear system as
\begin{equation}\label{eqn:MultiplicativeSolution}
\prod_{j \in V_i} x_j = (-1)^{b_i}
\end{equation}
for $i = 1, \ldots, m$ where $x_j \in \{\pm 1\}$.  We recall the definition of the solution group $\Gamma(A, b)$ associated to the system $Ax = b$.  The idea is to interpret \eqref{eqn:MultiplicativeSolution} as the relations of a group with generators $x_1, \ldots, x_n$ and a generator $J$ used to place the role of $-1$.  More precisely, we make the following definition.

\begin{defn}
Given an $m \times n$ linear system as above, let $\Gamma(A, b)$ denote the group generated by $u_1, \ldots, u_n, J$ with relations
\begin{enumerate}
  \item $u_j^2 = J^2 = 1$ for $j = 1, \ldots, n$,
  \item $u_j u_k = u_k u_j$ for $j, k \in V_i$ and $i = 1, \ldots, m$,
  \item $u_j J = J u_j$ for $j = 1, \ldots, n$, and
  \item $\prod_{j \in V_i} u_j = J^{b_i}$ for $i = 1, \ldots, m$.
\end{enumerate}
We call $\Gamma(A, b)$ the \emph{solution group} associated to the linear system $Ax = b$.
\end{defn}

For $i = 1, \ldots, m$, let
\[ S_i = \{ x \in \{\pm 1\}^n : \prod_{j \in V_i} x_j = (-1)^n \text{ and } x_j = 1 \text{ for  }j \notin V_i \}. \]
We associate a synchronous game to $Ax = b$ as follows:

\begin{defn}
Suppose $Ax = b$ is an $m \times n$ linear system over $\mathbb{Z}/2$ and $b \in (\mathbb{Z}/2)^n$.
The synchronous BCS game associated to $A x = b$, denoted $\operatorname{synBCS}(A, b)$, is given as follows:
\begin{enumerate}
  \item the input set is $\mathcal{I} = \{1,\ldots, m\}$;
  \item the output set is $\mathcal{O} = \{\pm 1 \}^n$;
  \item given input $(i,j)$, Alice and Bob win on output $(x,y)$ if $x \in S_i$, $y \in S_j$, and for all $k \in V_i \cap V_j$, $x_k = y_k$.
\end{enumerate}
\end{defn}

Let $\mathcal{A} \cong \mathrm{C}^*(\bb F(m, 2^n))$ denote the universal C$^*$-algebra generated by projections $e_{i, x}$ for $i = 1, \ldots, m$ and $x \in \{\pm 1\}^n$ subject to the relations $\sum_x e_{i, x} = 1$ for all $i = 1, \ldots, m$.  The following result gives a relationship between correlations in $C_{qc}^s(m, 2^n)$ and the structure of the group $\Gamma(A, b)$.

\begin{thm}\label{thm:surjection}
Suppose every column of $A$ contains a non-zero entry.  Then there is a surjective *-homomorphism $\pi : \mathcal{A} \rightarrow \mathrm{C}^*(\Gamma(A, b)) / \langle J + 1 \rangle,$ where $\langle J+1 \rangle$ denotes the ideal generated by $J+1$, given by
\begin{equation}\label{eqn:surjection}
  \pi(e_{i, x}) = \begin{cases} \prod_{j \in V_i} \chi_{x_j}(u_j) & x \in S_i \\ 0 & x \notin S_i, \end{cases}
\end{equation}
where $\chi_{x_j}(u_j)$ denotes the spectral projection of $u_j$ at the point $x_j$.

Moreover, the map $\tau \mapsto \tau \circ \pi$ is a bijection from the set of tracial states on $\mathrm{C}^*(\Gamma(A, b)) / \langle J + 1 \rangle$ to the set of tracial states $\tau'$ on $\mathcal{A}$ satisfying $\tau'(e_{i, x}e_{j, y}) = 0$ whenever Alice and Bob lose on outputs $(x, y)$ given inputs $(i, j)$.
\end{thm}

\begin{proof}
First we show that the formula for $\pi$ given in \eqref{eqn:surjection} defines a *-homomorphism on $\mathcal{A}$.  Note that since $\{u_j : j \in V_i\}$ is a set of commuting self-adjoint unitaries, $\pi(e_{i, x})$ is defined and is a projection for each $i$ and $x$.  Moreover, for $i = 1, \ldots, m$, in the algebra $\mathrm{C}^*(\Gamma(A, b)) / \langle J + 1 \rangle$,
\[ (-1)^{b_i} = \prod_{j \in V_i} u_j = \prod_{j \in V_i} (\chi_{+1}(u_j) - \chi_{-1}(u_j)) = \sum_{x \in \{\pm 1\}^{V_i}} \prod_{j \in V_i} x_j \chi_{x_j}(u_j). \]
Moreover, note that if $x \in \{\pm 1\}^{V_i}$ and $\prod_{j \in V_i} x_j \neq (-1)^{b_i}$, then
\[ \prod_{j \in V_i} x_j \chi_{x_j}(u_j) = - \prod_{j \in V_i} u_j \chi_{x_j}(u_j) = - \prod_{j \in V_i} x_j \chi_{x_j}(u_j) \]
and hence $\prod_{j \in V_i} x_j \chi_{x_j}(u_j) = 0$.  Combining these calculations, we have
\[ (-1)^{b_i} = \sum_{x \in S_i} \prod_{j \in V_i} x_j \chi_{x_j}(u_j) \]
and hence
\[ \sum_{x \in \{\pm 1\}^n} \pi(e_{i, x}) = \sum_{x \in S_i} \prod_{j \in V_i} \chi_{x_j}(u_j) = (-1)^{b_i} \sum_{x \in S_i} \prod_{j \in V_i} x_j \chi_{x_j}(u_j) = 1. \]
Thus the desired *-homomorphism $\pi$ exists.

To see $\pi$ is surjective, fix $k \in \{1, \ldots, m\}$.  As the $k$th column of $A$ contains a non-zero entry, there is an $i \in \{1, \ldots, m\}$ such that $k \in V_i$.  Note that
\begin{align*}
 u_k &= (\chi_{+1}(u_k) - \chi_{-1}(u_k)) \sum_{x \in S_i} \prod_{j \in V_i} \chi_{x_j}(u_j) \\
       &= \sum_{x \in S_i, x_k = 1} \prod_{j \in V_i} \chi_{x_j}(u_j) - \sum_{x \in S_i, x_k = -1} \prod_{j \in V_i} \chi_{x_j}(u_j) \\
       &= \sum_{x \in S_i, x_k = 1} \pi(e_{v, x}) - \sum_{x \in S_i, x_k = -1} \pi(e_{v, x}).
\end{align*}
As $\mathrm{C}^*(\Gamma(A, b)) / \langle J + 1 \rangle$ is generated by $u_1, \ldots, u_m$, the result follows.

We next work to prove the claim about traces.  As $\pi$ is surjective, the induced map on traces is injective.  To see surjectivity, let $\tau'$ be a trace on $\mathcal{A}$ such that $\tau'(e_{i, x} e_{j, y}) = 0$ if $x \notin S_i$, $y \notin S_j$, or there is a $k \in V_i \cap V_j$ such that $x_k \neq y_k$.  Define
\[ \mathcal{N} = \{ a \in \mathcal{A} : \tau'(a^*a) = 0 \} \]
and note that $\mathcal{N}$ is an ideal in $\mathcal{A}$.  We first show
\begin{enumerate}
  \item if $x \notin S_i$, then $e_{i, x} \in \mathcal{N}$,
  \item if $x_k \neq y_k$ for some $k \in V_i \cap V_j$, then $e_{i, x}e_{f, j} \notin \mathcal{N}$, and
  \item if $k \in V_i \cap V_j$, then $\displaystyle \sum_{x \in S_i} x_k e_{i, x} - \sum_{y \in S_j} y_k e_{i, x} \in \mathcal{N}$.
\end{enumerate}
First, if $x \notin S_i$, then $\tau'(e_{i, x}^*e_{i, x}) = \tau'(e_{i, x}e_{i, x}) = 0$ by the assumptions on $\tau'$.  Also, if $x_k \neq y_k$ for some $k \in V_i \cap V_j$, then
\[ \tau'((e_{i, x}e_{j, y})^*(e_{i, x}e_{j, y})) = \tau'(e_{j, y} e_{i, x} e_{j, y}) = \tau'(e_{i, x} e_{j, y}) = 0 \]
by the assumptions on $\tau'$.  For the final claim, fix $k \in V_i \cap V_j$.  Then
\[ \tau'(x_k y_k e_{i, x} e_{j, y}) = \begin{cases} \tau'(e_{i, x} e_{j, y}) & x_k = y_k \\ 0 & x_k \neq y_k \end{cases} \]
by (2) above.  Also,
\[ \sum_{x \in S_i} \tau'(e_{i, x}) = \sum_{x \in S_j} \tau'(e_{j, y}) = 1 \]
by (1) above.  Now,
\begin{align*}
 \tau' \left( \left(\sum_{x \in S_i} x_k e_{i, x} - \sum_{y \in S_j} y_k e_{i, x} \right)^* \left(\sum_{x \in S_i} x_k e_{i, x} - \sum_{y \in S_j} y_k e_{i, x}  \right) \right) \\
  = \sum_{x \in S_i} \tau'(e_{i, x}) + \sum_{y \in S_j} \tau'(e_{j, y}) - 2\sum_{x \in S_i, y \in S_j} \tau'(e_{i, x} e_{j, y}) = 0
\end{align*}
which proves (3).

Fix $k \in \{1, \ldots, n \}$.  Since the $j$th column of $A$ is non-zero, there is an $i \in \{1, \ldots, m \}$ such that $k \in V_i$.  Define $v_k \in \mathcal{A}/\mathcal{N}$ by
\[ v_k = \sum_{x \in S_i} x_k e_{i, x}. \]
by condition (3) above, the $v_k$ is independent of the choice of $i$.  Note that $v_k$ is a self-adjoint unitary in $\mathcal{A}/\mathcal{N}$ and if $k, \ell \in V_i$ for some $i = 1 \ldots m$, then $v_k v_\ell = v_\ell v_k$.  Finally for $i = 1, \ldots, m$, since the projections $e_{i, x}$ are orthogonal, we have
\[ \prod_{k \in V_i} v_k = \prod_{k \in V_i} \sum_{x \in S_i} x_k e_{i, x} = \sum_{x \in S_i} \left(\prod_{k \in V_i} x_k \right) e_{i, x} = (-1)^{b_i}. \]
It follows that there is a group homomorphism $\rho : \Gamma(A, b) \rightarrow U(\mathcal{A} / \mathcal{N})$ given by $\rho(u_k) = v_k$ and $\rho(J) = - 1$.  Now, $\rho$ induces a *-homomorphism, still denoted $\rho$, from $\mathrm{C}^*(\Gamma(A, b)) / \langle J + 1 \rangle$ to $\mathcal{A}/\mathcal{N}$.

Let $q : \mathcal{A} \rightarrow \mathcal{A}/\mathcal{N}$ denote the quotient map.  Since $\tau'$ vanishes on $\mathcal{N}$ by (1) and (2) above, there is a trace $\bar{\tau}'$ on $\mathcal{A} / \mathcal{N}$ such that $\bar{\tau}' \circ q = \tau'$.  Define a trace $\tau$ on $\mathrm{C}^*(\Gamma(A, b)) / \langle J + 1 \rangle$ by $\tau = \bar{\tau}' \circ \rho$.  By construction, $\rho(\pi(e_{i, x})) = q(e_{i, x})$ for all $i$ and $x$ and hence $\rho \circ \pi = q$.  Now, $\tau \circ \pi = \bar{\tau}' \circ \rho \circ \pi = \bar{\tau}' \circ q = \tau'$.  This completes the proof.
\end{proof}

\begin{cor}\label{cor:SynStratsAndReps}
Let $Ax = b$ be a linear system.
\begin{enumerate}
  \item $\operatorname{synBCS}(A, b)$ has a perfect qc-strategy if and only if $J \neq 1$ in $\Gamma(A, b)$,
  \item $\operatorname{synBCS}(A, b)$ has a perfect qa-strategy if and only if there is representation $\Gamma(A, b) \rightarrow \mathcal{R}^\omega$ such that $\rho(J) \neq 1$, and
  \item $\operatorname{synBCS}(A, b)$ has a perfect q-strategy if and only if there is a finite dimensional representation $\rho : \Gamma(A, b) \rightarrow U(\bb M_d)$ such that $\rho(J) \neq 1$.
\end{enumerate}
\end{cor}

\begin{proof}
We may assume no column of $A$ is identically zero.  Assume $A$ is an $m \times n$ linear system.

We first prove (1).  If $\operatorname{synBCS}(A, b)$ has a perfect qc-strategy $p(x, y|i, j) \in C_{qc}^s(m, 2^n)$, there is a trace $\tau$ on $\cal A$ such that
\[ p(x, y|i, j) = \tau(e_{i,x} e_{j, y}) \quad \text{ for all } i, j, x, y. \]
By Theorem \ref{thm:surjection}, there is a trace $\tau'$ on $\mathrm{C}^*(\Gamma(A, b)) / \langle J + 1 \rangle$ such that $\tau' \circ \pi = \tau$.  In particular, $\mathrm{C}^*(\Gamma(A, b)) / \langle J + 1 \rangle$ is non-zero.  Hence $J + 1 \neq 2$ in $\mathrm{C}^*(\Gamma(A, b))$ and $J \neq 1$ in $\Gamma(A, b)$.

Conversely, suppose $J \neq 1$ in $\Gamma(A, b)$.  As $J$ is central, $\langle J \rangle \cong \mathbb{Z}/2$ is a normal subgroup of $\Gamma(A, b)$.  There is a conditional expectation $E : \mathrm{C}^*(\Gamma(A, b)) \rightarrow \mathrm{C}^*(\langle J \rangle) \cong \mathbb{C}^2$ determined by $E(s) = s$ for $s \in \{1, J\}$ and $E(s) = 0$ for $s \in \Gamma(A, b) \setminus \{1, J\}$.  Let $\chi : \mathrm{C}^*(\langle J \rangle) \rightarrow \mathbb{C}$ be the character defined by $\chi(J) = -1$.  Then $\chi \circ E$ is a trace on $\mathrm{C}^*(\Gamma(A, b))$.  As $(\chi \circ E)(J + 1) = 0$ and $J + 1 \geq 0$, the trace $\chi \circ E$ vanishes on the ideal $\langle J + 1 \rangle \subseteq \mathrm{C}^*(\Gamma(A, b))$ and hence induces a trace $\tau$ on $\mathrm{C}^*(\Gamma(A, b)) / \langle J + 1 \rangle$.  Now, the trace $\tau \circ \pi$ on $\cal A$ is a trace where $\pi$ is the surjection in Theorem \ref{thm:surjection}.  We define a qc-correlation by \[ p(x, y|i, j) = \tau(\pi(e_{i, x}e_{j, y})) \quad \text{ for all } i, j, x, y. \]
By Theorem \ref{thm:surjection}, $\big(p(x, y|i, j)\big)$ is a perfect qc-strategy.

For (2) and (3), we let $\cal B$ denote either $\cal R^\omega$ or $\mathbb{M}_d$.  Suppose $\rho : \Gamma(A, b) \rightarrow U(\mathcal{B})$ is a group homomorphism such that $\rho(J) \neq 1$.  Let $q$ denote the spectral projection of $\rho(J)$ corresponding to the eigenvalue $-1$.  As $J \neq 1$, we have $q \neq 0$.  As $J$ is central in $\Gamma(A, b)$, the projection $q$ commutes with the image of $\rho$.  Now, $q \rho(\cdot)$ is a unitary representation of $\Gamma(A, b)$ on $U(q\cal B q)$ and $q \rho(J) = - q$.  When $\cal B = \mathbb{M}_d$, $q \cal B q \cong \mathbb{M}_{d'}$ for some $d' \geq 1$, and when $\cal B = \cal R^\omega$, $q \cal B q \cong \mathcal{R}^\omega$.  Hence after replacing $\cal B$ with $q \cal B q$ and $\rho$ with $q \rho(\cdot)$, we may assume $\rho(J) = -1$.  Now $\rho$ induces a *-homomorphism $\mathrm{C}^*(\Gamma(A, b)) \rightarrow \cal B$ vanishing on $J + 1$ and hence induces a *-homomorphism
\[ \cal A \overset{\pi}{\longrightarrow} \mathrm{C}^*(\Gamma(A, b)) / \langle J + 1 \rangle \rightarrow \cal B. \]
The trace on $\cal B$ defines a trace on $\cal A$ which in turn defines a winning q-strategy when $\cal B$ is finite dimensional and a winning qa-strategy when $\cal B = \cal R^\omega$.

Now suppose $\operatorname{synBCS}(A, b)$ has a perfect qa-strategy.  As in Theorem \ref{thm:surjection}, there is a trace $\tau$ on $\mathrm{C}^*(\Gamma(A, b)) / \langle J + 1 \rangle$ which factors through the trace on $\cal R^\omega$.  The GNS representation of $\tau$ induces a representation of $\mathrm{C}^*(\Gamma(A, b)) / \langle J + 1 \rangle \rightarrow \mathcal{R}^\omega$ which in turn induces a representation $\rho : \Gamma(A, b) \rightarrow \cal R^\omega$ with $\rho(J) = -1$.  Similarly, if $\operatorname{synBCS}(A, b)$ has a perfect q-strategy, one produces a representation of $\Gamma(A, b)$ in the same way using a finite dimensional algebra in place of $\cal R^\omega$.
\end{proof}

The following result is due to Slofstra in \cite{Slofstra17}.

\begin{thm}\label{thm:Slofstra'sGroup}
There is a linear system $A x = b$ such that there is a representation $\rho : \Gamma(A, b) \rightarrow U(\mathcal{R}^\omega)$ such that $\rho(J) \neq 1$ but for every finite dimensional representation $\rho_0 : \Gamma(A, b) \rightarrow U(\mathbb{M}_d)$, $\rho(J) = 1$.
\end{thm}

Combining Theorem \ref{thm:Slofstra'sGroup} with Corollary \ref{cor:SynStratsAndReps} provides a synchronous game which has a perfect qa-strategy but no perfect q-strategy.  Hence we have the following strengthening of Slofstra's result in \cite{Slofstra17}.

\begin{cor}
For sufficiently large $m$ and $n$, we have $C_q^s(m,2^n) =C_{qs}^s(m, 2^n) \neq C_{qa}^s(m, 2^n)$.  In particular, for sufficiently large $m,n$, $C_q^s(m, 2^n)=C_{qs}^s(m, 2^n)$ is not closed.
\end{cor}

\begin{remark} If $C_{qs}(m, 2^n)$ or $C_q(m,2^n)$ was closed, then their subsets of synchronous elements would be closed. Since $C_q^s(m,2^n) = C_{qs}^s(m, 2^n)$, the above result implies Slofstra's result \cite{Slofstra17} that $C_q(m,2^n)$ and $C_{qs}(m,2^n)$ are not closed, for sufficiently large $m,n$. It is not clear if there is any direct proof of the converse, i.e., that the non-closure of the sets $C_{qs}(m,2^n)$ and $C_q(m,2^n)$ implies that their synchronous subsets are not closed.
\end{remark}

\section{Separating quantum independence numbers of graphs}


In this section we prove that there exists a graph $G$ for which $\alpha_{q}(G) < \alpha_{qa}(G)$.  Recall from Section 2 that for $t \in \{q, qa, qc \}$, the independence number $\alpha_t(G)$ is the largest $c \geq 1$ for which the graph homomorphism game $K_c \rightarrow \cl G$ has a perfect $t$-strategy.

%


First let us recall from \cite[Section 6]{graphisom} the graph $G_{A, b}$ defined for a linear system $Ax = b$ over $\mathbb{Z}/2$.

\begin{defn} Suppose $Ax = b$ is an $m \times n$ linear system over $\mathbb{Z}/2$ and $b \in (\mathbb{Z}/2)^n$.
  Define a graph $G_{A, b}$ with the following data:
  \begin{enumerate}
    \item the vertices of $G_{A, b}$ are pairs $(i,x)$ where $i \in \{1,\ldots, m\}$ and $x \in S_i$;
    \item there is an edge between distinct vertices $(i, x)$ and $(j, y)$ if and only if there exists some $k \in V_i \cap V_j$ for which $x_i \neq y_j$; that is, $x$ and $y$ are inconsistent solutions.
  \end{enumerate}
\end{defn}


\begin{lemma}\label{lemma:IndepNumberIsomInvariant}
Suppose $t \in \{q, qa, qc\}$.  If $G$ and $H$ are finite graphs and $G \cong_t H$ then $\alpha_t(G) = \alpha_t(H)$.
\end{lemma}

\begin{proof}
Let $V = V(G) \cup V(H)$. It suffices to show that if $G \cong_t H$, then whenever $\alpha_t(G) \geq c$, we also have $\alpha_t(H) \geq c$.  As $\alpha_t(G) \geq c$, there is a C$^*$-algebra $\cal A$, a tracial state $\tau_\cal A$ on $\cal A$, and projections $e_{i, v} \in \cal A$ for $i = 1, \ldots, c$ and $v \in V(G)$ such that $\sum_v e_{i, v} = 1$ for all $i = 1, \ldots, c$ and $\tau(e_{i, v} e_{j, w}) = 0$ whenever $(v, w) \in E(G)$.  If $t = q$, we may assume $\tau_A$ factors through a finite dimensional algebra and if $t = qa$, we may assume $\tau_\cal A$ is amenable.

Similarly, since $G \cong_t H$, there is a C$^*$-algebra $\cal B$, a tracial state $\tau_\cal B$ on $\cal B$, and projections $q_{v, w} \in \cal B$ for $v, w \in V$ such that $\sum_{w \in V} q_{v, w} = 1$ for all $v \in V$ and such that if $v, w \in V(G)$ and $x, y \in V(H)$ with $\text{rel}(v,w) \neq \text{rel}(x,y)$ then $\tau_B(q_{vx}q_{wy}) = 0$.  (Note that there are other relations in the graph isomorphism game; these are the only ones we will need to use here.)  Again we choose $\tau_\cal B$ to factor through a finite dimensional algebra if $t = q$ and we choose $\tau_\cal B$ to be amenable if $t = qa$.

For $i = 1, \ldots, c$ and $x \in V(H)$, define
\[ f_{i, x} = \sum_{v \in V(G)} e_{i, v} \otimes f_{v, x} \in \cal A \otimes \cal B. \]
Then each $f_{i, x}$ is a projection and for all $i = 1, \ldots, c$, we have $\sum_x f_{i, x} = 1$.  If $x, y \in V(H)$ and $(x, y) \in E(H)$, then
\[ \tau_\cal A \otimes \tau_\cal B(f_{i, x} f_{j, y}) = \sum_{v, w \in V(G)} \tau_\cal A(e_{i, v} e_{j, w}) \tau_\cal B(f_{v, x} f_{w, y}). \]
For $v, w \in V(G)$, if $(v, w) \in E(G)$, then $\tau_\cal A(e_{i, v} e_{j, w}) = 0$, and if $(v, w) \notin E(G)$, then $\tau_\cal B(f_{v, x} f_{w, y}) = 0$.  Hence the projections $f_{i, x} \in \cal A \otimes \cal B$ and the trace $\tau_\cal A \otimes \tau_\cal B$ determine a perfect qc-strategy for the graph homomorphism game from $K_c$ to $\cl H$.  If $\tau_A$ and $\tau_B$ factor through finite dimensional algebras, so does $\tau_A \otimes \tau_B$.  If $\tau_A$ and $\tau_B$ are amenable, so is $\tau_A \otimes \tau_B$.  Hence in all cases, $\alpha_t(H) \geq c$.
\end{proof}

It is shown in \cite[Theorem 3.7]{OrtizPaulsen} that for $t \in \{q,qa,qc\}$ and graphs $G$, $H$ and $K$, if $G \stackrel{t}{\rightarrow} H$ and $H \stackrel{t}{\rightarrow} K$ then
$G \stackrel{t}{\rightarrow} K$. This leads to the following corollary.

\begin{cor}
If $t \in \{q, qa, qc \}$ and $G$ is a finite graph, then $\alpha_t(G) \leq \chi_t(\cl{G})$.
\end{cor}

\begin{proof}
  Suppose that $\alpha_t(G) = c$. By definition, there is a $t$-homomorphism
  $K_c \stackrel{t}{\rightarrow} \cl{G}$. If $\chi_{t}(\cl G) =d$ then there is a $t$-homomorphism, $\cl G \stackrel{t}{\rightarrow} K_d$. Since qa-homomorphisms are closed under composition, there is a $t$-homomorphism $K_c \stackrel{t}{\rightarrow} K_d$ which implies that $\chi_{t}(K_c) \leq d$.  As noted in Section 2 above, $\chi_{t}(K_c) = c$ and hence $c \leq d$ as claimed.
\end{proof}

In the case $t = q$, the following result appears as Theorem 6.2 in \cite{graphisom}.

\begin{thm}
Suppose $t \in \{q, qa, qc \}$ and let $Ax = b$ be an $m \times n$ linear system.  The following are equivalent:
\begin{enumerate}
  \item the game $\operatorname{synBCS}(A, b)$ has a winning $t$-strategy;
  \item $G_{A, b} \cong_t G_{A, 0}$;
  \item $\alpha_t(G_{A, b}) = m$.
\end{enumerate}
\end{thm}

\begin{proof}
(1) $\Rightarrow$ (2): Suppose that we have a winning $t$-strategy for the $\operatorname{synBCS(A, b)}$.  Fix a C$^*$-algebra $\cal B$, a faithful trace $\tau \in \cal B$, and projections $e_{i, x} \in \cal B$ for $i = 1, \ldots, m$ and $x \in \{\pm 1\}^n$ such that $\sum_x e_{i, x} = 1$ for all $i$, $e_{i, x} = 0$ if $x \notin S_i$, and $e_{i, x} e_{j, y} = 0$ if there is a $k \in V_i \cap V_j$ with $x_k \neq y_k$.  If $t = q$, we assume $\cal B$ is finite dimensional and if $t = qa$, we assume $\cal B \subseteq \cal R^\omega$.  Let $\cal G$ be the isomorphism game for $(G_{A, b}, G_{A, 0})$ and let $\cal A(\cal G)$ denote the algebra associated to $\cal G$ as defined in Section 2 above. It suffices to construct a unital *-homomorphism $\pi: \cal A(\cal G) \rightarrow \cal B$.

Let $S_i^0 \subseteq \{\pm 1\}^n$ denote the set of solutions to the $i$th equation of the linear system $Ax = 0$ and let $S_i^1 \subseteq \{\pm 1\}^n$ denote the set of solutions to the $i$th equation of the linear system $Ax = b$.  Given $x, y \in \{\pm 1\}^n$, let $xy \in \{\pm 1\}^n$ denote the pointwise product of $x$ and $y$.  Note that if $x \in S_i^0$ and $y \in S_i^0$, then $xy \in S_i^1$.  Moreover, for $x \in S_i^1$, the map $S_i^0 \rightarrow S_i^1$ given by $y \mapsto xy$ is a bijection.

For $(i, x) \in V(G_{A, b})$ and $(j, y) \in V(G_{A, 0})$, define
\[ q_{(i, x), (j, y)} = \begin{cases} e_{i, xy} & i = j \\ 0 & i \neq j \end{cases} \]
and note that each $q_{(i, x), (j, y)}$ is a projection.  For $(i, x) \in V(G_{A, b})$, we have
\[ \sum_{(j, y) \in V(G_{A, 0})} q_{(i, x), (j, y)} = \sum_{j=1}^n \sum_{y \in S_j^0} q_{(i, x), (j, y)} = \sum_{y \in S_i^0} e_{i, xy} = \sum_{z \in S_i^1} e_{i, z} = 1. \]
A similar computation shows that for all $(j, y) \in V(G_{A, 0})$, we have
\[ \sum_{(i, x) \in V(G_{A, b})} q_{(i, x), (j, y)} = 1. \]

We need to show that for all $(i, x), (i', x') \in V(G_{A, b})$ and $(j, y), (j', y') \in V(G_{A, 0})$, the implication
\[ q_{(i,x),(j,y)} q_{(i',x'),(j',y')} \neq 0 \quad \Rightarrow \quad \operatorname{rel}((i,x), (i',x')) = \operatorname{rel}((j,y),(j',y')) \]
holds.  To this end, suppose $q_{(i,x),(j,y)} q_{(i',x'),(j',y')} \neq 0$.  Then $i = j$, $i' = j'$, and $e_{i, xy} e_{i', x'y'} \neq 0$.  We consider several cases.

Suppose first $i = i'$.  Then we have $xy = x'y'$.  If $x = x'$, then $y = y'$ and we have both $(i, x) = (i', x')$ and $(j, y) = (j', y')$ so the right hand side of the implication holds in the case.  Conversely, if $x \neq x'$ and $y \neq y'$, then $(i, x) \neq (i', x')$ and $(j, y) \neq (j', y')$.  Note also that since $i = i'$, $x$ and $x'$ are necessarily inconsistent solutions so that $(i, x)$ and $(i', x')$ are adjacent.  Similar reasoning shows $(j, y)$ and $(j', y')$ are adjacent.  Hence the right hand side of the implication holds.

Now assume $i \neq i'$ so that, in particular, $(i, x) \neq (i', x')$.  If $(i, x)$ and $(i', x')$ are adjacent, there is a $k \in V_i \cap V_{i'}$ such that $x_k \neq x'_k$.  On the other hand, as $e_{i, xy} e_{i', x'y'} \neq 0$, we know $x_k y_k = (xy)_k = (x'y')_k = x'_k y'_k$.  Therefore, $y_k \neq y'_k$ so that $(i, y)$ and $(i', y')$ are adjacent.  Finally, suppose $(i, x)$ and $(i', x')$ are not adjacent.  Then $x_k = x'_k$ for all $i \in V_i \cap V_{i'}$.  Again since $e_{i, xy} e_{i', x'y'} \neq 0$, we also know $x_k y_k = x'_k y'_k$ for all $k \in V_i \cap V_{i'}$ and therefore $y_k = y'_k$ for all $k \in V_i \cap V_{i'}$ so that $(j, y)$ and $(j', y')$ are not adjacent.  This covers all cases.

Now, the projections $q_{(i, x), (j, y)} \in \cal B$ define a unital *-representation $\pi: \cal A(\cal G) \rightarrow \cal B$ and it follows that $G_{A, b} \cong_t G_{A, 0}$.

(2) $\Rightarrow$ (3): Suppose that $G_{A, b} \cong_t G_{A, 0}$.  By Lemma \ref{lemma:IndepNumberIsomInvariant}, it suffices to show that $\alpha_t(G_{\cal F_0}) = m$.  The map $f : \cl{G_{A, 0}} \to \{1, \ldots, m \}: (i,x) \mapsto i$ is an $m$-colouring of $\cl{G_{A, 0}}$.  Indeed, suppose are $(i, x)$ and $(j, y)$ are distinct vertices in $\cl{G_{A, 0}}$ with $f(i, x) = f(j, y)$.  Then $i = j$ and hence $x \neq y$.  That is, there is some $k \in V_i$ such that $x_k \neq y_k$ and thus there is no edge between $(i, x)$ and $(j, y)$ in $\cl{G_{A, 0}}$.

For each $i = 1, \ldots, m$, the vector $x_0 = (1, \ldots, 1)$ is in $S_i \subseteq \{\pm 1\}^n$ for the system $Ax = 0$.  Hence for $i, j = 1, \ldots, m$, there is no edge between the vertices $(i, x_0)$ and $(j, x_0)$ in $\cl{G_{A, 0}}$ and we have $\alpha(G_{A, 0}) \geq m$.  Now,
\begin{align*}
  m \geq \chi(\cl{G_{A, 0}}) \geq \chi_t(\cl{G_{A, 0}}) \geq \alpha_t(G_{A, 0}) \geq \alpha(G_{A, 0}) \geq m,
\end{align*}
and $\alpha_t(G_{A, 0}) = m$.

(3) $\Rightarrow$ (1):  Suppose $\alpha_t(G_{A, b}) = m$.  Then the graph homomorphism game from $K_m$ to $\overline{G_{A, b}}$ has a perfect $t$-strategy.  Fix a C$^*$-algebra $\cal A$ with a faithful trace $\tau$ and projections $e_{k, i, x} \in \cal A$ for $i = 1, \ldots, m$, $v \in V(G_{A, b})$ such that
\begin{enumerate}
  \item $\displaystyle \sum_{i=1}^m \sum_{x \in S_i} e_{k, i, x} = 1$ for all $1 \leq k \leq m$, and
  \item $\tau(e_{k, i, x} e_{\ell, j, y}) = 0$ if there is an edge between $(i, x)$ and $(j, y)$ in $G_{A, b}$.
\end{enumerate}
If $t = q$, we may assume $\cal A$ is finite dimensional and if $t = qa$, we may assume $\cal A = \cal R^\omega$.

Define for $i = 1, \ldots, m$ and $j \in V_i$,
\[ v_{i,j} = \sum_{k=1}^m \sum_{x \in S_k} x_j e_{i, k, x} \]
and note that $a_{i, j}$ is a self-adjoint unitary since by since the $\sum_{k, x} e_{i, k, x} = 1$ and for all $j$, we have $x_j \in \{\pm 1\}$.
Also, $v_{i, j}$ and $v_{i, k}$ commute for all $i, j, k = 1, \ldots m$.

For all $i,k$ and $j \in V_i \cap V_k$,
\[ \tau(v_{i, j} v_{k, j}) = \sum_{p, q = 1}^m \sum_{x \in S_p, y \in S_q} x_j y_j \tau(e_{i, p, x} e_{k, q, y}). \]
Note that when $x_j \neq y_j$, there is an edge between $(p, x)$ and $(q, y)$ in $G_{A, b}$ and hence $\tau(e_{i, p, x} e_{k, q, y}) = 0$.  Moreover, when $x_j = y_j$, $x_j y_j = 1$.  Hence we have
\[ \tau(v_{i, j} v_{k, j}) = \sum_{p, q = 1}^m \sum_{x \in S_p, y \in S_q} \tau(e_{i, p, x} e_{k, q, y}) = 1. \]
Now, for $i, k = 1, \ldots, m$ and $j \in V_i \cap V_k$, we have
\[ \tau((v_{i, j} - v_{k, j})^*(v_{i, j} - v_{k, j})) = 2 - 2 \tau(v_{i, j} v_{k, j}) = 0, \]
and hence $v_{i, j} = v_{k, j}$ as $\tau$ is faithful.

Given $j = 1, \ldots, n$, define $w_j = v_{i, j}$ if $j \in V_i$ for some $i = 1, \ldots, m$ and $w_j = 1$ otherwise.  By the previous paragraph, this is well-defined.  If $j, k = 1, \ldots, n$ and there is an $i = 1, \ldots, m$ with $j, k \in V_i$, then $w_j = v_{i, j}$ and $w_k = v_{i, k}$ commute.  Moreover, for each $i = 1, \ldots, m$,
\[ \prod_{j \in V_i} w_j = \prod_{j \in V_i} v_{i, j} = \sum_{k=1}^m \sum_{x \in S_{\ell}} \prod_{j \in V_i} x_j e_{i,k, x} = (-1)^{(b_i)}. \]
Hence there is a representation $\rho : \Gamma(A, b) \rightarrow U(\mathcal{A})$ such that $\rho(u_i) = w_i$ and $\rho(J) = -1$ for all $i = 1, \ldots, n$.  By Corollary \ref{cor:SynStratsAndReps}, the game $\operatorname{synBCS}(A, b)$ has a perfect $t$-strategy which proves (1).
\end{proof}

\begin{cor}
  There exists a graph $G$ for which $\alpha_{qa}(G) > \alpha_{q}(G)$.
\end{cor}

\begin{cor}
There exist graphs $G$ and $H$ for which $G \cong_{qa} H$ but $G \not\cong_q H$.
\end{cor}
{\em Acknowledgements.}  The authors would like to thank N. Manor, A. Mehta and A. Winter for their valuable comments.


\bibliographystyle{alpha}
\bibliography{synBCS}
\end{document}